\documentclass[12pt]{article}

\usepackage{amsthm,amsfonts,amsbsy,amssymb,amsmath,amscd}
\usepackage[colorlinks=true]{hyperref}

\begin{document}
\title{Volumes of Arithmetic Okounkov Bodies}
\author{Xinyi Yuan}
\maketitle

\theoremstyle{plain}
\newtheorem{thm}{Theorem}[section]
\newtheorem*{conj}{Conjecture}
\newtheorem*{notation}{Notation}
\newtheorem{cor}[thm]{Corollary}
\newtheorem*{corr}{Corollary}
\newtheorem{lem}[thm]{Lemma}
\newtheorem{pro}[thm]{Proposition}
\newtheorem{definition}[thm]{Definition}
\newtheorem*{thmm}{Theorem}

\theoremstyle{remark} \newtheorem*{remark}{Remark}
\theoremstyle{remark} \newtheorem*{example}{Example}

\newcommand{\RR}{\mathbb{R}}
\newcommand{\QQ}{\mathbb{Q}}
\newcommand{\CC}{\mathbb{C}}
\newcommand{\ZZ}{\mathbb{Z}}
\newcommand{\FF}{\mathbb{F}}
\newcommand{\PP}{\mathbb{P}} 

\newcommand{\fp}{\mathbb{F}_p}
\newcommand{\fwp}{{\mathbb{F}_{\wp}}}

\newcommand{\rank}{\mathrm{rank}}        
\newcommand{\vol}{\mathrm{vol}}          
\newcommand{\divv}{\mathrm{div}}         
\newcommand{\Spec}{\mathrm{Spec}}        
\newcommand{\ord}{\mathrm{ord}}          

\newcommand{\lb}{\mathcal{L}}            
\newcommand{\nb}{\mathcal{N}}
\newcommand{\fb}{\mathcal{F}}
\newcommand{\eb}{\mathcal{E}}
\newcommand{\tb}{\mathcal{T}}
\newcommand{\ob}{\mathcal{O}}
\newcommand{\ib}{\mathcal{I}}
\newcommand{\jb}{\mathcal{J}}
\newcommand{\ab}{\mathcal{A}}
\newcommand{\bb}{\mathcal{B}} 

\newcommand{\mb}{\overline{M}}

\newcommand{\hhat}{\hat h^0}
\newcommand{\Hhat}{\widehat H^0}

\newcommand{\xb}{\mathcal{X}}             

\newcommand{\lbb}{\overline{\mathcal{L}}}             
\newcommand{\mbb}{\overline{\mathcal{M}}}
\newcommand{\nbb}{\overline{\mathcal{N}}}
\newcommand{\ebb}{\overline{\mathcal{E}}}
\newcommand{\tbb}{\overline{\mathcal{T}}}
\newcommand{\obb}{\overline{\mathcal{O}}}
\newcommand{\abb}{\overline{\mathcal{A}}}
\newcommand{\bbb}{\overline{\mathcal{B}}}

\newcommand{\lbt}{\widetilde{\mathcal{L}}}
\newcommand{\mbt}{\widetilde{\mathcal{M}}}

\newcommand{\chern}{\hat{c}_1}           
\newcommand{\height}{h_{\lbb}}           
\newcommand{\cheight}{\hat{h}_{\lb}}     

\newcommand{\lnorm}[1]{\|#1\|_{L^2}}     
\newcommand{\supnorm}[1]{\|#1\|_{\mathrm{sup}}}    
\newcommand{\chil}{\chi_{_{L^2}}}
\newcommand{\chisup}{\chi_{\sup}}

\newcommand{\amp}{\widehat{\mathrm{Amp}}(X)}   
\newcommand{\nef}{\widehat{\mathrm{Nef}}(X)}   
\newcommand{\eff}{\widehat{\mathrm{Eff}}(X)}   
\newcommand{\pic}{\widehat{\mathrm{Pic}}(X)}   
\newcommand{\bigx}{\widehat{\mathrm{Big}}(X)}

\tableofcontents

\section{Introduction}

This paper is an improvement and simplification of Yuan \cite{Yu}.
In particular, by a simpler method, it proves the identity in \cite[Theorem
A]{Yu} without taking the limit $p\to\infty$.
The method also simplifies the proofs of Moriwaki \cite{Mo}, 
where more general arithmetic linear series were treated. 

Recall that \cite{Yu} explored a way to construct arithmetic Okounkov bodies from an arithmetic line
bundle, inspired by the idea of Okounkov \cite{Ok1, Ok2} and Lazarsfeld--Musta\c
t\v a \cite{LM} in the geometric case. 
Theorem A of \cite{Yu} asserts that the volumes of the Okounkov bodies
approximate the volume of the arithmetic line bundle. 
The main result of this paper asserts an exact identity before taking the limit.

On the other hand, Boucksom--Chen \cite{BC} initiated a different way to construct Okounkov bodies in the arithmetic setting. From the proof of the main result in this paper, it is easy to recognize the similarity of the constructions in \cite{BC} and in \cite{Yu}. We will make a comparison at the end of this paper.

In the following, we recall the construction of \cite{Yu} and state our main
result. We use exactly the same notations as in \cite{Yu} throughout this paper.

\subsection*{Volume of an arithmetic line bundle}

Let $X$ be an arithmetic variety of dimension $d$, i.e., a $d$-dimensional
integral scheme, projective and flat over $\Spec(\ZZ)$.
For any hermitian line bundle $\lbb=(\lb, \|\cdot\|)$ over $X$,
denote 
$$\Hhat(X,\lbb)= \{s\in H^0(X,\lb): \|s\|_{\sup} \leq 1\}$$ 
and 
$$\hhat(X,\lbb)= \log \# \Hhat(X,\lbb).$$
Define the volume to be
$$\vol(\lbb)=\limsup_{m\rightarrow \infty}  \frac{\hhat(X,m\lbb)}{m^d/d!}.$$
Recall that a line bundle $\lbb$ is said to be \emph{big} if $\vol(\lbb)>0$.

Note that Chen \cite{Ch} proved that ``limsup=lim'' in the definition of
$\vol(\lbb)$. The result was also obtained in \cite{Yu} by the construction of
the Okounkov body.

\subsection*{Arithmetic Okounkov Body}

Assume that $X$ is normal with smooth generic fiber $X_{\QQ}$.
Denote by $X\to \Spec(O_K)$ the Stein factorization of $X\to \Spec(\ZZ)$.
Then $K$ is the algebraic closure of $\QQ$ in the function field of $X$. 
For any prime ideal $\wp$ of $O_K$, denote by $\fwp=O_K/\wp$ the residue field, 
and by $N_\wp$ the cardinality of $\fwp$. 

Let 
$$X \supset Y_1 \supset \cdots \supset Y_d$$
be a flag on $X$, where each $Y_i$ is a regular irreducible closed subscheme of
codimension $i$ in $X$. 
Assume that $Y_1$ is the fiber $X_{\fwp}$ of $X$ above some prime ideal $\wp$ of $O_K$, 
and $Y_d\in Y_1(\fwp)$ is a rational point.

Define a valuation map 
$$\nu_{Y.}=(\nu_1, \cdots, \nu_d): H^0(X,\lb)-\{0\} \rightarrow \ZZ^d$$
with respect to the flag $Y.$ as follows. For any nonzero $s\in H^0(X,\lb)$, we
first set $\nu_1(s)=\ord_{Y_1}(s)$. 
Let $s_{Y_1}$ be a section of the line bundle $\ob(Y_1)$ with zero locus $Y_1$.
Then 
$s_{Y_1}^{\otimes(-\nu_1(s))}s$ is nonzero on $Y_1$, and let $s_1=\left.
\left(s_{Y_1}^{\otimes(-\nu_1(s))}s\right)\right|_{Y_1}$ be the restriction. 
Set $\nu_2(s)=\ord_{Y_2}(s_1)$. 
Continue this process on the section $s_1$ on $Y_2$, we can define $\nu_3(s)$
and thus $\nu_4(s), \cdots, \nu_d(s)$.

For any hermitian line bundle $\lbb$ on $X$, denote 
$$v_{Y.}(\lbb)=\nu_{Y.}(\Hhat(X,\lbb)-\{0\})$$
to be the image in $\ZZ^d$.
The \emph{arithmetic Okounkov body} $\Delta_{Y.}(\lbb)$ of $\lbb$ is defined to
be the closure of 
$\Lambda_{Y.}(\lbb)=\displaystyle\bigcup_{m\geq 1}\frac{1}{m^d} v_{Y.}(m\lbb)$
in $\RR^d$. It is a bounded convex subset of $\RR^d$ if non-empty. 
The following is the main result of this paper.

\begin{thm} \label{main}
If $\lbb$ is big, then
$$\vol(\Delta_{Y.}(\lbb)) \log N_\wp =\frac{1}{d!}\vol(\lbb). $$
\end{thm}

In \cite{Yu}, the author restricted to the case $\fwp=\fp$, and could only prove that the left-hand side
converges to the right-hand side as $p\to \infty$. 
The problem was caused by a large accumulated error term in the estimation.
The solution of this problem is found by the author in the preparation of \cite{YZ}, a recent
joint work of the author with Tong Zhang.
The key is Theorem \ref{key}, which puts the filtration together and makes a
more accurate control of the error term.

\

{\footnotesize
\noindent\textit{Acknowledgments.}  
This paper is based on a result of lattice points in a joint work of
the author with Tong Zhang. The author would like to thank Tong Zhang for many inspiring discussions.
The author is very grateful for Huayi Chen who provided a way to extend
the result of lattice points from the field of rational numbers to general number fields. 
}

\section{Lattice points in a filtration}

In this section, we state and prove the result on lattice points in Theorem \ref{key}. 
It is the key for Theorem \ref{main}.

Fix a number field $K$. By a \emph{normed $O_K$-module}, we mean a pair
$(M,\{\|\cdot\|_\sigma\}_\sigma)$ consisting of a locally free $O_K$-module $M$ 
of finite rank, and a collection $\{\| \cdot \|_\sigma\}_\sigma$ of $\CC$-norms
$\| \cdot \|_\sigma$ on $M\otimes_ {\sigma}\CC$,  indexed by $\sigma:K\hookrightarrow \CC$  and
invariant under the complex conjugation.

Let $\overline M=(M,\{\|\cdot\|_\sigma\}_\sigma)$ be normed $O_K$-module.  
Define
$$
\Hhat(\overline M)=\{m \in M: \| m \|_\sigma \le 1, \quad \forall \sigma\}, 
$$
and
$$
\hhat(\overline M)= \log \# \Hhat(\overline M).
$$
Denote by 
$$O_K\langle \Hhat(\overline M) \rangle, \quad \ZZ\langle \Hhat(\overline M) \rangle$$
respectively the $O_K$-submodule and the $\ZZ$-submodule of $M$ generated by $\Hhat(\overline M)$. 
For any $\alpha\in \RR$, denote 
$$\overline M(\alpha)=(M,\{e^{-\alpha}\|\cdot\|_\sigma\}_\sigma).$$

For a normed $O_K$-module $\overline L$ of rank one,  define
$$
\widehat\deg (\overline L)= \log \#(L/sO_K) -\sum_\sigma  \log\|s\|_\sigma.
$$
Here $s\in L$ is any non-zero element, and the definition is independent of the choice of $s$.
It is just the usual arithmetic degree of a hermitian line bundle over $\Spec(O_K)$.

It is clear that the dual $\overline L^\vee$ gives a natural normed $O_K$-module of rank one, 
and that the tensor product $\overline M\otimes \overline L$ is a natural normed $O_K$-module.
The identity of the tensor is given by the trivial normed $O_K$-module $\overline O_K=(O_K, \{|\cdot|_\sigma\}_\sigma)$, where $|\cdot|_\sigma$ is just the usual absolute value. 

In some literatures, a normed $O_K$-module is also called a metrized vector bundle over $\Spec(O_K)$.
A normed $O_K$-module of rank one is also called a metrized line bundle over $\Spec(O_K)$, or just
a hermitian line bundle over $\Spec(O_K)$.

The innovation of this paper is the application of the following result.

\begin{thm} \label{key}
Let $K$ be a number field. 
Let $\overline M$ be a normed $O_K$-module,
and $\overline L_0,\overline L_1, \cdots, \overline L_n$ be a sequence of normed $O_K$-modules
of rank one. Here $\overline L_0$ is the trivial normed $O_K$-module of rank one.
Denote
$$\alpha_i=\widehat\deg(\overline L_i), \quad 
r_i=\rank_{O_K} O_K\langle\widehat H^0(\overline M\otimes \overline L_i^\vee)\rangle, 
\quad i=0,\cdots, n.$$
Assume that
$$0=\alpha_0\leq \alpha_1\leq \cdots \leq \alpha_n.$$
Then there is a constant $C>0$ depending only on $K$ such that 
\begin{eqnarray*}
\widehat h^0(\overline M)
&\geq& 
\sum_{i=1}^{n} r_i(\alpha_i- \alpha_{i-1})
-C(r_0 \log r_0+r_0), \\
\widehat h^0(\overline M) 
&\leq&
 \widehat h^0(\overline M\otimes \overline L_n^\vee)+
\sum_{i=1}^{n} r_{i-1} (\alpha_{i}-\alpha_{i-1}) +C(r_0 \log r_0+r_0).
\end{eqnarray*}

\end{thm}

The case $K=\QQ$ is equivalent to \cite[Proposition 2.3]{YZ}, which follows from 
the successive minima of Gillet--Soul\'e \cite{GS}. 
The extension to the general $K$ is provided by Huayi Chen in a private communication to the author. 
We start with a simple lemma in the following.

\begin{lem} \label{kq}
Denote $\kappa=[K:\QQ]$ and denote by $D_K$ the discriminant of $K$.
\begin{itemize}
	\item[(1)] For any hermitian normed $O_K$-module $\overline L$ of rank one, we have
$$ \hhat(\overline L) > \widehat\deg(\overline L)- \kappa\log2- \frac12 \log|D_K|.$$
	\item[(2)] There is a constant $\delta\geq 0$ depend only on $K$ such that
$$
 {\rm rank }_{O_K}O_K\langle\Hhat(\overline M(-\delta))\rangle
\leq \frac{1}{\kappa}
 {\rm rank }_{\ZZ}\ZZ\langle\Hhat(\overline M)\rangle
\leq  {\rm rank }_{O_K}O_K\langle\Hhat(\overline M)\rangle
$$
for any normed $O_K$-module $\mb$. 
\end{itemize}

\end{lem}

\begin{proof}

The result in (1) follows from Minkowski's theorem. The second inequality of (2) follows from the inclusion
$$
\ZZ\langle\Hhat(\overline M)\rangle  \subset O_K\langle\Hhat(\overline M)\rangle. 
$$
For the first inequality of (2), fix a basis $(a_1,\cdots, a_\kappa)$ of $O_K$ over $\ZZ$. 
Set
$$
\delta=\max \{\log|a_i|_\sigma: \ i=1,\cdots, \kappa, \ \sigma: K\hookrightarrow \CC\}.
$$
Then 
\begin{eqnarray*}
&& O_K\langle\Hhat(\overline M(-\delta))\rangle \\
&\subset & 
\ZZ\langle a_1\Hhat(\overline M(-\delta)),\cdots, a_\kappa\Hhat(\overline M(-c))\rangle\\
&\subset & 
\ZZ\langle\Hhat(\overline M)\rangle.
\end{eqnarray*}
The inequality also follows. 
 
\end{proof}

\begin{proof}[Proof of Theorem \ref{key}]

If $K=\QQ$, then $\overline L_i$ is isometric to $\overline\ZZ(\alpha_i)$, and thus 
$$\overline M\otimes \overline L_i^\vee=\overline M(-\alpha_i).$$
The theorem is equivalent to the inequalities for $\hhat$ in \cite[Proposition 2.3]{YZ}. 
Next we consider the general case.

Denote $\kappa=[K:\QQ]$ and $\displaystyle c=\log2+ \frac{1}{2\kappa} \log|D_K|$.
Denote $\beta_i=\alpha_i/\kappa$ for convenience. 
View $\mb$ as a normed $\ZZ$-module, and apply \cite[Proposition 2.3]{YZ} to the sequence 
$$0=\beta_0\leq \beta_1+c\leq \cdots \leq \beta_n+c.$$
We obtain
\begin{eqnarray*}
\widehat h^0(\overline M) 
&\leq&
 \widehat h^0(\overline M(-\beta_n-c))+
\sum_{i=1}^{n} r_{i-1}' (\beta_{i}-\beta_{i-1})  + O(r_0 \log r_0).
\end{eqnarray*}
Here 
$$r_i'=\rank_{\ZZ} \ZZ\langle\widehat H^0(\overline M(-\beta_i-c))\rangle, 
\quad i=0,\cdots, n-1.$$

We claim that $r_i'\leq \kappa r_i$ by our special choice of $c$. 
It simply implies the second equality we need to prove. 
By Lemma \ref{kq} (1), $\Hhat(\overline L_i^\vee(\beta_i+c))\neq 0$.
Since 
$$ \overline M\otimes \overline L_i^\vee
=\overline M(-\beta_i-c) 
 \otimes \overline L_i^\vee(\beta_i+c),$$
we have
$$r_i'=\rank_{\ZZ} \ZZ\langle\widehat H^0(\overline M(-\beta_i-c))\rangle
\leq \rank_{\ZZ} \ZZ\langle\widehat H^0(\overline M\otimes \overline L_i^\vee)\rangle
\leq \kappa r_i.$$
The last inequality follows from Lemma \ref{kq} (2).
It proves the claim. 

As for the first inequality, apply \cite[Proposition 2.3]{YZ} to $\mb(\delta+c)$ viewed as a normed $\ZZ$-module, and
the sequence 
$$0=\beta_0\leq \beta_1\leq \cdots \leq \beta_n.$$
Here $\delta$ is as in Lemma \ref{kq} (2).
We obtain
\begin{eqnarray*}
\widehat h^0(\overline M(\delta+c)) 
 \geq  \sum_{i=1}^{n} r_i''(\beta_i- \beta_{i-1})+O( r_0 \log r_0+r_0).
\end{eqnarray*}
Here 
$$
r_i''= \rank_{\ZZ} \ZZ\langle\widehat H^0(\overline M(-\beta_i+\delta+c))\rangle.
$$
Lemma \ref{kq} (1) gives $\Hhat(\overline L_i(-\beta_i+c))\neq 0$, and thus
$$
r_i''
\geq \rank_{\ZZ} \ZZ\langle\widehat H^0(\overline M\otimes \overline L_i^\vee(\delta))\rangle
\geq \kappa r_i.
$$
Here the last inequality follows from Lemma \ref{kq} (2).
The proof is finished by the basic inequality
$$
\widehat h^0(\overline M)  \geq \widehat h^0(\overline M(\delta+c)) 
- \kappa r_0(\delta+c+\log3). 
$$ 
See \cite[Proposition 2.1]{YZ} or the original form \cite[Proposition 4]{GS}. 

\end{proof}

\section{The arithmetic Okounkov body}

In this section, we prove Theorem \ref{main}, and compare our construction with that of \cite{BC}.

\subsection*{Proof of the main theorem}

Recall that $X$ is a normal arithmetic variety with smooth generic fiber,
geometrically irreducible over $O_K$. 
Recall that we have the flag
$$X \supset Y_1 \supset Y_2 \supset \cdots \supset Y_d$$
of regular vertical subvarieties. 
We have assume that $Y_1=X_{\fwp}$ for some
prime ideal $\wp$ of $O_K$, and $Y_d\in Y_1(\fwp)$ is a rational point.

By \cite[Proposition 2.5]{Yu}, if $\lbb$ is a big hermitian line bundle on $X$, 
$$\lim_{m\rightarrow \infty} \frac{\# v_{Y.}(m\lbb)}{m^d}=\vol(\Delta_{Y.}).$$
Hence, Theorem \ref{main} is reduced to the following result, which strengthens
\cite[Theorem 2.6]{Yu}.

\begin{thm}\label{comparison level m}
For any $\lbb\in \pic$, 
\begin{eqnarray*}
\lim_{m\rightarrow \infty} \left|\frac{\# v_{Y.}(m\lbb)}{m^d} \log N_\wp -
\frac{\hhat(X,m\lbb)}{m^d}\right|=0.
\end{eqnarray*}

\end{thm}

Now we prove the theorem. 
Recall that 
$$v(\lbb)=v_{Y.}(\lbb)=\nu_{Y.}(\Hhat(X,\lbb))$$
is the image in $\ZZ^d$ of
the valuation map 
$$\nu=\nu_{Y.}=(\nu_1, \cdots, \nu_d): H^0(X,\lb)-\{0\} \rightarrow \ZZ^d$$
defined by the flag
$$X \supset Y_1 \supset Y_2 \supset \cdots \supset Y_d.$$

The flag 
$$Y_1 \supset Y_2 \supset \cdots \supset Y_d$$
 on the ambient variety $Y_1$ induces a valuation map 
$$\nu^\circ=(\nu_2, \cdots,\nu_d):  H^0(X,\lb_{\fwp})-\{0\} \rightarrow
\ZZ^{d-1}$$
of dimension $d-1$ in the geometric case. 
Similarly, we have a valuation map $\nu^\circ$ on the line bundle $(\lb\otimes \wp^i)|_{Y_1}$ for any $i\in\ZZ$. 
There are natural isomorphisms $(\lb\otimes \wp^i)|_{Y_1}\cong \lb_{\fwp}$,
which are compatible with the valuations.

The valuations $\nu$ and $\nu^\circ$ are compatible in the sense that 
$$\nu(s)=\left(\nu_1(s), \nu^\circ( (s_{Y_1}^{\otimes(-\nu_1(s))}s) |_{Y_1} )\right),\quad
s\in \Hhat(X,\lbb)-\{0\}.$$ 
Here $s_{Y_1}$ is the section of $\ob(Y_1)$ defining $Y_1$.

Decompose $v(\lbb)=\nu_{Y.}(\Hhat(\lbb))$ according to the first component in
$\ZZ^d$.
We have 
$$
v(\lbb)=\coprod_{i\geq 0} \left(i,\ \nu^\circ (\Hhat(\lbb\otimes\wp^i)|_{Y_1} )
\right).
$$
Here $\wp$ is naturally a hermitian line bundle on $\Spec(O_K)$, 
endowed with the metric induced from the trivial metric of $O_K$ via the inclusion $\wp\subset O_K$. 
It is also viewed as a hermitian line bundle on $X$ by pull-back.

By the geometric case in \cite{LM}, 
$$
\#\nu^\circ (\Hhat(\lbb\otimes\wp^i)|_{Y_1} )
=\dim_\fwp \fwp\langle\Hhat(\lbb\otimes\wp^i)|_{Y_1} \rangle.
$$
Here $\fwp\langle\Hhat(\lbb\otimes\wp^i)|_{Y_1} \rangle$ denotes the $\fwp$-vector
subspace of 
$H^0(\lb\otimes\wp^i|_{Y_1})$ generated by $\Hhat(\lbb\otimes\wp^i)|_{Y_1}$. 
The reduction map gives an injection
$$
H^0(\lb\otimes\wp^i)/\wp H^0(\lb\otimes\wp^i)\longrightarrow H^0(\lb\otimes\wp^i|_{Y_1}).
$$
It follows that 
$$
\dim_\fwp \fwp\langle\Hhat(\lbb\otimes\wp^i)|_{Y_1} \rangle
=\rank_{O_K} O_K\langle\Hhat(\lbb\otimes\wp^i)  \rangle.
$$
Here $ O_K\langle\Hhat(\lbb\otimes\wp^i)  \rangle$ denotes the $O_K$-submodule of 
$H^0(\lb\otimes\wp^i)$ generated by $\Hhat(\lbb\otimes\wp^i)$. 
Therefore, 
\begin{eqnarray} \label{one}
\# v(\lbb)=\sum_{i\geq 0} \rank_{O_K} O_K\langle\Hhat(\lbb\otimes\wp^i)  \rangle.
\end{eqnarray}
Note that it is essentially a finite sum. 

For any $i\geq 0$, denote $\overline L_i= \wp^{-i}$. Then 
$$\widehat\deg (\overline L_i)=i \log N_\wp.$$
Apply Theorem \ref{key} to the normed $O_K$-module 
$$\overline M=(H^0(\lb), \{\|\cdot\|_{\sigma,\sup}\}_\sigma)$$ 
and the sequence $\{\overline L_i\}_{i\geq 0}$. 
It is easy to obtain
\begin{eqnarray} \label{two}
\widehat h^0( \lbb)
= \sum_{i\geq 0} \rank_{O_K} O_K\langle\Hhat(\lbb\otimes\wp^i)  \rangle \cdot \log N_\wp
+O(h^0(\lb_\QQ)\log h^0(\lb_\QQ)).
\end{eqnarray}

Compare (\ref{one}) and (\ref{two}). We have
$$
\# v(\lbb)\cdot \log N_\wp=\widehat h^0( \lbb) +O(h^0(\lb_\QQ)\log h^0(\lb_\QQ)).
$$
Replace $\lbb$ by $m\lbb$. We obtain Theorem \ref{comparison level m}.

\subsection*{Construction of Boucksom--Chen}

In \cite{BC}, Boucksom and Chen constructed Okounkov bodies in a very general arithmetic setting. 
Here we briefly compare their construction with our construction in the setting of this paper.

We first recall the construction of \cite{BC}. 
Let $(X, \lbb)$ be as above. 
Denote by $Z_1=X_K$ the generic fiber. 
Fix a flag 
$$Z_1\supset Z_2 \supset \cdots \supset Z_d $$
on the generic fiber $Z_1$. 
It gives an Okounkov body $\Delta_{Z.}(\lb_K)\subset\RR^{d-1}$ of the generic fiber $\lb_K$.

Restricted to the generic fiber, we have a valuation 
$$
\nu_{Z.}^\circ=(\nu_2', \cdots, \nu_d'): \Hhat(X,\lbb)-\{0\} \rightarrow \ZZ^{d-1}.
$$
It gives a convex body $\Delta_{Z.}^{\circ}(\lbb)$ contained in $\Delta_{Z.}(\lb_K)$. 

More generally, for any $t\in \RR$, denote $\lbb(-t)=(\lb, e^{t}\|\cdot\|)$. 
Then we have a convex body $\Delta_{Z.}^{\circ}(\lbb(-t))$ of $\lbb(-t)$ contained in $\Delta_{Z.}(\lb_K)$.
The sequence $\{\Delta_{Z.}^{\circ}(\lbb(-t))\}_{t\in \RR}$ is decreasing in $t$. 
Define the incidence function $G_{Z.}:\Delta_{Z.}(\lb_K) \to \RR$ by
$$
G_{Z.}(x)= \sup\{t\in\RR: x\in \Delta_{Z.}^{\circ}(\lbb(-t)) \}. 
$$
Its graph gives a $(d+1)$-dimensional convex body 
$$
\Delta_{Z.}(\lbb)=\{ (x, t)\in  \Delta_{Z.}(\lb_K)\times \RR:
0\leq t\leq G_{Z.}(x)  \}.
$$ 
This is the Okounkov body constructed in \cite{BC}. 
One main result of \cite{BC} asserts that 
$$\vol(\Delta_{Z.}(\lbb))  =\frac{1}{d!}\vol(\lbb). $$

Go back to our construction. 
Recall that we have a flag 
$$X \supset Y_1 \supset \cdots \supset Y_d$$
on $X$. On the special fiber $Y_1=X_{\fwp}$, we have a flag 
$$Y_1 \supset Y_2\supset \cdots \supset Y_d.$$
It gives a valuation of the line bundle $\lb_{\fwp}$ on $Y_1$. 
Thus we get an Okounkov body $\Delta_{Y.}(\lb_{\fwp})$ in $\RR^{d-1}$. 

Restricted to the special fiber, we have a valuation 
$$
\nu_{Y.}^\circ=(\nu_2, \cdots, \nu_d): \Hhat(X,\lbb)-\{0\} \rightarrow \ZZ^{d-1}.
$$
It gives a convex body $\Delta_{Y.}^{\circ}(\lbb)$ contained in $\Delta_{Y.}(\lb_\fwp)$. 

The hermitian line bundle $\lbb\otimes \wp^t$ gives a convex body
$\Delta_{Y.}^{\circ}(\lbb\otimes \wp^t)$ in $\RR^{d-1}$ for any $t\in \ZZ$.
By passing to tensor powers, extend the definition of $\Delta_{Y.}^{\circ}(\lbb\otimes \wp^t)$ to
any $t\in\QQ$. 
The sequence $\{\Delta_{Y.}^{\circ}(\lbb\otimes \wp^t)\}_{t\in \QQ}$ is decreasing in $t$. 
Define the incidence function $G_{Y.}:\Delta_{Y.}(\lb_\fwp) \to \RR$ by
$$
G_{Y.}(x)= \sup\{t\in\QQ: x\in \Delta_{Y.}^{\circ}(\lbb\otimes \wp^t) \}. 
$$
Its graph gives a $(d+1)$-dimensional convex body 
$$
\Delta_{Y.}(\lbb)=\{ (x, t)\in  \Delta_{Y.}(\lb_\fwp)\times \RR:
0\leq t\leq G_{Y.}(x)  \}.
$$ 
It recovers our previous construction. 

Hence, we see that the similarity of these two constructions. The construction of \cite{BC} is archimedean in nature, 
while the construction of \cite{Yu} is non-archimedean in nature.

\end{document}